\documentclass{amsart}
\usepackage{graphicx,hyperref,url,tikz-cd,amssymb,mathtools,wrapfig,pgfplots}

\usepackage{color}

\newtheorem{theorem}{Theorem}[section]
\newtheorem{lemma}[theorem]{Lemma}
\newtheorem{prop}[theorem]{Proposition}
\newtheorem{Co}[theorem]{Corollary}
\theoremstyle{definition}

\theoremstyle{remark}
\newtheorem{remark}[theorem]{Remark}
\numberwithin{equation}{section}

\newcommand{\CC}{\mathbb{C}}

\newcommand{\MM}{\mathbb{M}}
\newcommand{\NN}{\mathbb{N}}

\newcommand{\RR}{\mathbb{R}}

\newcommand{\bh}{B(\mathcal{H})}
\newcommand{\cl}[1]{\mathcal{#1}}

\newcommand\Ac{{\mathcal{A}}}
\newcommand\Afr{{\mathfrak A}}
\newcommand\Hc{{\mathcal{H}}}
\newcommand\hti{{\widetilde h}}
\newcommand\lspan{\mathrm{span}\,}
\newcommand\thetat{{\widetilde\theta}}
\newcommand\xt{{\widetilde x}}

\begin{document}
\title{The Delta Game}

\author[K.~Dykema]{Ken Dykema}
\address{Department of Mathematics, Texas A\& M University, College Station, TX}
\email{kdykema@math.tamu.edu}

\author[V.~I.~Paulsen]{Vern I.~Paulsen}
\address{Institute for Quantum Computing and Department of Pure Mathematics, University of Waterloo,
Waterloo, ON, Canada N2L 3G1}
\email{vpaulsen@uwaterloo.ca}

\author[J.~Prakash]{Jitendra Prakash}
\address{Institute for Quantum Computing and Department of Pure Mathematics, University of Waterloo,
Waterloo, ON, Canada N2L 3G1}
\email{jprakash@uwaterloo.ca}

\thanks{The second and third authors were supported in part by NSERC}

\begin{abstract}
We introduce a game related to the $I_{3322}$ game and analyze a constrained value function for this game over various families of synchronous quantum probability densities.
%.present an analysis on the synchronous value of Cleve's delta game.
\end{abstract}
\date{\today (Last revised)}

\maketitle
%\tableofcontents

\section{Introduction}
A subject of a great deal of current research has been the study of various mathematical models for what should constitute the set of quantum correlations beginning with Tsirelson \cite{Ts1, Ts2} and continuing with \cite{junge2011etal, Fritz, ozawa2013, PT, dykema2015, paulsen2016, Sl1, Sl2}. In particular, the {\it Tsirelson conjectures} ask whether or not several different mathematical models for these conditional quantum probabilities yield the same sets of probability densities. Whether or not two of these models yield the same sets of densities is now known to be equivalent to {\it Connes' embedding problem} \cite{junge2011etal, Fritz, ozawa2013}.

Recently, the work of W. Slofstra \cite{Sl1, Sl2} has shown that these different models generally yield different sets. Equality of the two sets of probabilities corresponding to Connes' problem is the only unresolved case.

One way to try and distinguish between these various sets of probability densities is by studying the values of finite input-output games. Indeed, this is the approach that Slofstra uses successfully. However, his games are rather large and his results rely on some very deep results in the theory of finitely presented groups.  So it is always interesting to see if any simpler games can illuminate these differences or if in fact their potentially different quantum values all coincide.  
%\section{Preliminaries}

Another topic of current interest is attempting to compute the quantum value of the $I_{3322}$ game and decide if it is actually attained over the standard set of quantum probability densities. This has been the subject of a great deal of research \cite{Froissart1981}, \cite{Collins2004}, \cite{Pal2010}, \cite{Vidick2011}. The game that we study is a simplification of the $I_{3322}$ game.

In \cite{dykema2015} the concept of the synchronous value of a finite input-output game was introduced. In this paper we compute the synchronous quantum values, for each of these different models, for a simplification of the $I_{3322}$ game that was brought to our attention by R. Cleve that we call the $\Delta$ game.

Our results show that the synchronous values of this game corresponding to the different mathematical models of  quantum probabilities that we consider all coincide.  However, we introduce an extra constraint to the value function that allows us to compute the value at each constraint parameter. In this case the value of the game is the supremum of this constrained value function over the constraint parameter. We show that although the four values of the game are all equal, only three of the four functions coincide.  

%This calculation implies that the ``vect" model for densities gives a different set of densities than the other quantum models.  This is not a new fact, but our results show that by regarding the value of a game as a function of a constraint parameter one can see a better separation of these models. 

We show that the graphs of the three quantum value function coincide, but that they lie strictly below the function that one obtains by computing values over the set of {\it vector correlations}. This gives another way to see that the set of 3 input 2 output vector correlations is strictly larger than the corresponding sets of probabilistic quantum correlations. It was already known fact that these sets differ at the 2 input 2 output level---for example using the CHSH game. In addition, earlier work of \cite{CMNSW} used quantum chromatic numbers to show that there is a 15 input 7 output game with a perfect strategy in the set of vector correlations, but no perfect strategy among the quantum correlations.

 The main new insight from our results is that the process of studying constrained value functions of games can allow one to see separation when the unconstrained value sees no separation. 

%The reason for our interest in these vector correlations, is that Originally, Tsirelson believed that one could show that all of the various models for quantum probability densities coincided by showing that they also coincided with these vectorial correlations. 

\section{Preliminaries}
Recall that a general two person finite input-output game $\mathcal{G}$, involves two noncommunicating players, Alice ($A$) and Bob ($B$), and a Referee ($R$). The game is described by $\mathcal{G}=(I_A,I_B,O_A,O_B,\lambda)$ where $I_A,I_B,O_A,O_B$ are nonempty finite sets, representing Alice's inputs, Bob's inputs, Alice's outputs and Bob's outputs, respectively, and with $\lambda: I_A\times I_B\times O_A\times O_B \rightarrow \{0,1\}$ a function.

% with \begin{align*}
%\lambda(v,w,i,j)=\begin{cases} 0, & \text{ means that the move is not allowed,} \\ 1, & \text{ means that the move is allowed}, \end{cases}
%\end{align*} where $v\in I_A, w\in I_B, i\in O_A$ and $j\in O_B$. 
For each {\it round} of the game,
Alice receives input $v\in I_A$ and Bob receives input $w\in I_B$ from the Referee and then Alice and Bob produce outputs $i\in O_A$ and $j\in O_B$, respectively. They win if $\lambda(v,w,i,j) =1$ and lose if $\lambda(v,w,i,j) =0$. The function $\lambda$ is called the {\it rule} or {\it predicate} function. 

%Henceforth we will assume that $I=I_A=I_B=\{0,1,...,n-1\}$ and $O=O_A=O_B=\{0,1,...,m-1\}$
Suppose that
Alice and Bob have a random way to produce outputs. This is informally what is meant by a {\it strategy}. If we observe a strategy over many rounds we will obtain conditional probabilities $(p(i,j|v,w))$, where $p(i,j|v,w)$ is the joint conditional probability that Alice outputs $i$ on input $v$ and Bob outputs $j$ on input $w$. For this reason, any tuple $(p(i,j|v,w))_{i \in O_A, j \in O_B, v \in I_A, w \in I_B}$ satisfying 
\[ p(i,j|v,w) \ge 0 \, \, \text{  and } \sum_{i \in O_A, j \in O_B} p(i,j|v,w) =1, \, \forall \,\, v \in I_A, w \in I_B,\]
will be called a {\it correlation}. 

 A correlation $(p(i,j|v,w))$ is called a \emph{winning} or {\it perfect} correlation for $\cl G$ if \begin{align*}
\lambda(v,w,i,j)=0 \Rightarrow p(i,j|v,w)=0,
\end{align*} that is, it produces disallowed outputs with zero probability.

If we also assume that the Referee chooses inputs according to a known probability distribution $\pi:I_A\times I_B\to [0,1]$, that is, 
  \begin{align*}
\pi(v,w)\geq 0 \quad\text{ and } \quad \sum_{(v,w)\in I_A\times I_B}\pi(v,w)=1,
\end{align*} 
then it is possible to assign a number to each correlation that measures the probability that Alice and Bob will win a round given their correlation.
The {\it value} of the correlation $p=(p(i,j|v,w))$, corresponding to the distribution $\pi$ on inputs,  is given by \begin{align*}
V(p, \pi) = \sum_{i,j,v,w}  \lambda(v,w,i,j) \pi(v,w) p(i,j|v,w).
\end{align*} Note that a perfect correlation always has value 1 and, provided that $\pi(v,w) >0$ for all $v$ and $w$ a correlation will have value 1 if and only if it is a perfect correlation.

The {\it value of the game} $\cl G$ with respect to a fixed probability density $\pi$ on the inputs over a given set $\cl F$ of correlations is given by \begin{align*}
\omega_{\cl F}(\cl G, \pi)  = \sup \{ V(p, \pi): p \in \cl F \}.
\end{align*}
Because the set of all correlations is a bounded set in a finite dimensional vector space, whenever $\cl F$ is a closed set, it will be compact and so this supremum over $\cl F$ will be attained.

A finite input-output game as above is called {\it synchronous} provided that $I_A=I_B:=I$, $O_A=O_B:= O$ and for all $v \in I$, $\lambda(v,v, i,j) =0$ whenever $i \ne j$. This condition can be summarized as saying that whenever Alice and Bob receive the same input then they must produce the same output.  A correlation $(p(i,j|v,w))$ is called {\it synchronous} provided that $p(i,j|v,v) = 0$ for all $ v \in I_A$ and for all $i \ne j$.  Note that when $\cl G$ is a synchronous game, then any perfect correlation must be synchronous.

In this paper we are interested in studying the $\Delta$ game, which is a synchronous game, and computing $\omega(\Delta, \cl F)$ as we let $\cl F$ vary over the various mathematical models for synchronous quantum correlations.  We now introduce these various models for quantum densities.

Recall that a set, $\{R_k\}_{k=1}^n$, of operators on some Hilbert space $\mathcal{H}$ is called a \emph{positive operator valued measure} (POVM) provided $R_k\geq 0$, for each $k$, and $\sum_{k=1}^nR_k=I$. Also a set of projections, $\{P_k\}_{k=1}^n$, on some Hilbert space $\mathcal{H}$ is called a \emph{projection valued measure} (PVM) provided $\sum_{k=1}^nP_k=I$. Thus every PVM is a POVM.

A \emph{quantum} correlation for a game $\mathcal{G}$ means that Alice and Bob have finite dimensional Hilbert spaces $\mathcal{H}_A$ and $\mathcal{H}_B$, respectively. For each input $v\in I$, Alice has a PVM -- $\{P_{v,i}\}_{i\in O}$ on $\mathcal{H}_A$, and similarly for each input $w\in I$, Bob has a PVM -- $\{Q_{w,j}\}_{j\in O}$ on $\mathcal{H}_B$. They also share a state $h\in\mathcal{H}_A\otimes \mathcal{H}_B$ ($\|h\|=1$) such that \begin{align*}
p(i,j|v,w)=\left\langle (P_{v,i}\otimes Q_{w,j})h,h\right\rangle.
\end{align*} The set of all $(p(i,j|v,w))$ arising from all choices of finite dimensional Hilbert spaces $\mathcal{H}_A,\mathcal{H}_B$, all PVMs and all states $h$ is called the set of \textit{quantum correlations} denoted by $C_q(n,m)$. 

Another family of correlations are the \emph{commuting quantum} correlations. In this case there is a single (possibly infinite dimensional) Hilbert space $\mathcal{H}$ and for each input $v\in I$, Alice has a PVM  $\{P_{v,i}\}_{i\in O}$, and similarly for each input $w\in I$, Bob has a PVM  $\{Q_{w,j}\}_{j\in O}$, satisfying $P_{v,i}Q_{w,j}=Q_{w,j}P_{v,i}$ (hence the name commuting). They share a state $h\in\mathcal{H}$ ($\|h\|=1$) such that \begin{align*}
p(i,j|v,w)=\left\langle (P_{v,i}Q_{w,j})h,h\right\rangle.
\end{align*} The set of all $(p(i,j|v,w))$ arising this way is denoted by $C_{qc}(n,m)$ and is called the set of \emph{commuting quantum correlations}. 

\begin{remark} In the above definitions one could replace the PVM's with POVM's throughout, and this is used as the definitions of these sets in many references. Since there are more POVM's then PVM's one might obtain larger sets, say $\widetilde{C}_q(n,m)$ and $\widetilde{C}_{qc}(n,m)$. But, in fact,        $\widetilde{C}_q(n,m)=C_q(n,m)$ and $\widetilde{C}_{qc}(n,m)= C_{qc}(n,m)$. The fact that $\widetilde{C}_q(n,m)= C_q(n,m)$ follows by a simple dilation trick. On the Hilbert space $\cl H_A$, one simply uses a Naimark dilation to enlarge the space to $\cl K_A$ and dilate the set of POVM's to a set of PVM's on $\cl K_A$. One similarly dilates Bob's POVM's to PVM's on $\cl K_B$ and then considers the tensor products of these PVM's on $\cl K_A \otimes \cl K_B$.  The proof that $\widetilde{C}_{qc}(n,m) = C_{qc}(n,m)$ is somewhat more difficult and can be found in \cite[Proposition
3.4]{Fritz}, and also as Remark 10 of \cite{junge2011etal}. A third proof appears in \cite{PT}. We shall sometimes refer to this as the {\em disambiguation} of the two possible definitions.
\end{remark}

\begin{remark}\label{remark-q-in-qc}
By Theorem 5.3 in \cite{paulsen2016}, $C_q(n,m)\subseteq C_{qc}(n,m)$, with $(p(i,j|v,w))\in C_q(n,m)$ if and only if $(p(i,j|v,w))\in C_{qc}(n,m)$ such that the Hilbert space $\mathcal{H}$ in its realization is finite dimensional. %We will concentrate on the correlation set $C_{qc}(n,m)$ in the rest of this paper.
\end{remark}

There is yet another correlation set denoted by $C_{vect}(n,m)$ that is often called the set of \textit{vector correlations}. It is the set of all $(p(i,j|v,w))$ such that $p(i,j|v,w)=\left\langle x_{v,i},y_{w,j}\right\rangle$ for sets of vectors $\{x_{v,i}:v\in I,i\in O\}, \{y_{w,j}:w\in I,j\in O\}$ in a Hilbert space $\mathcal{H}$ and a unit vector $h\in\mathcal{H}$, which satisfy \begin{enumerate}
\item $x_{v,i}\perp x_{v,j}$ and $y_{w,i}\perp y_{w,j}$ for all $i\neq j$ in $O$.
\item $\sum_{i\in O}x_{v,i}=h=\sum_{j\in O}y_{w,j}$ for all $v,w\in I$.
\item $\left\langle x_{v,i},y_{w,j} \right\rangle\geq 0$ for all $v,w\in I$ and $i,j\in O$.
\end{enumerate} These correlations have been studied at other places in the literature, see for
example \cite{NGHA} where they are referred to as {\it almost quantum correlations} and they can be interpreted as the first level of the NPA hierarchy \cite{NPA}.

The above correlation sets are related in the following way \begin{align}\label{order1}
C_q(n,m)\subseteq C_{qc}(n,m)\subseteq C_{vect}(n,m)\subset \RR^{n^2m^2},
\end{align} for all $n,m\in\NN$ and they are all convex sets. It is known that the sets $C_{qc}(n,m)$ and $C_{vect}(n,m)$ are closed sets in $\RR^{n^2m^2}$. Set $C_{qa}(n,m)=\overline{C_{q}(n,m)}$ so that \begin{align}\label{order2}
C_q(n,m)\subseteq C_{qa}(n,m)\subseteq C_{qc}(n,m),
\end{align} for all $n,m\in \NN$. W. Slofstra \cite{Sl2} recently proved that there exists an $n$ and $m$ such that $C_q(n,m)$ is not a closed set. Hence $C_{q}(n,m)$ is in general a proper subset of $C_{qa}(n,m)$, but whether or not they are different for all values of $n,m$ is unknown. It also remains an open question to determine whether $C_{qa}(n,m)=C_{qc}(n,m)$ for all $n$ and $m$ or not.  In \cite{junge2011etal}, it was proven that if Connes' embedding problem is true then $C_{qa}(n,m)=C_{qc}(n,m)$ for all $n$ and $m$. The converse was proven in \cite{ozawa2013}. Thus we know that  $C_{qa}(n,m) = C_{qc}(n,m), \, \forall n,m$  is equivalent to Connes' embedding conjecture.

For $t\in\{q, qa, qc, vect\}$, let $C_t^s(n,m)$ denote the subset of all synchronous correlations. The synchronous sets $C_t^s(n,m)$ are also convex for $t\in\{q, qa, qc, vect\}$. The set of synchronous commuting quantum correlations, $C_{qc}^s(n,m)$, may be characterized in the following way.

Let $\mathcal{A}$ be a unital C$^*$-algebra. Recall that a linear functional $\varphi:\mathcal{A}\rightarrow\CC$ is called a \textit{tracial state} if $\varphi$ is positive, $\varphi(1)=1$, and $\tau(ab)=\tau(ba)$ for all $a,b\in\mathcal{A}$. 
%Finally, a tracial state $\tau:\cl A\to \CC$ is called \textit{faithful} if $\tau(a)>0$ whenever $a\in \cl A$ is nonzero and positive.

\begin{theorem}[Theorem 5.5, \cite{paulsen2016}]\label{qcsyn}
Let $(p(i,j|v,w))\in C_{qc}^s(n,m)$ be realized with PVMs $\{P_{v,i}:v\in I\}_{i\in O}$ and $\{Q_{w,j}:w\in I\}_{j\in O}$ in some $\bh$ satisfying $P_{v,i}Q_{w,j}=Q_{w,j}P_{v,i}$ and with some unit vector $h\in\mathcal{H}$ so that $p(i,j|v,w)=\left\langle P_{v,i}Q_{w,j} h,h \right\rangle$. 
Then \begin{enumerate}
\item $P_{v,i}h=Q_{v,i}h$ for all $v\in I, i\in O$;
\item $p(i,j|v,w)=\langle (P_{v,i}P_{w,j})h,h\rangle=\langle (Q_{w,j}Q_{v,i})h,h\rangle=p(j,i|w,v)$;
\item Let $\mathcal{A}$ be the $C^*$-algebra in $\bh$ generated by the family $\{P_{v,i}:v\in I, i\in O\}$ and define $\tau:\mathcal{A}\rightarrow\CC$ by $\tau(X)=\left\langle Xh,h\right\rangle$. Then $\tau$ is a tracial state on $\mathcal{A}$ and $p(i,j|v,w)=\tau(P_{v,i}P_{w,j})$.
\end{enumerate} Conversely, let $\mathcal{A}$ be a unital $C^*$-algebra equipped with a tracial state $\tau$ and with $\{e_{v,i}:v\in I, i\in O\}\subset \mathcal{A}$ a family of projections such that $\sum_{i\in O}e_{v,i}=1$ for all $v\in I$. Then $(p(i,j|v,w))$ defined by $p(i,j|v,w)=\tau(e_{v,i}e_{w,j})$ is an element of $C_{qc}^s(n,m)$. That is, there exists a Hilbert space $\mathcal{H}$, a unit vector $h\in \mathcal{H}$ and mutually commuting PVMs $\{P_{v,i}:v\in I\}_{i\in O}$ and $\{Q_{w,j}:w\in I\}_{j\in O}$ on $\mathcal{H}$ such that \begin{align*}
p(i,j|v,w)=\langle (P_{v,i}Q_{w,j})h,h\rangle=\langle (P_{v,i}P_{w,j})h,h\rangle=\langle (Q_{w,j}Q_{v,i})h,h\rangle
\end{align*}
\end{theorem}

% When $\mathcal{A}=\MM_n$, there is a unique tracial state denoted by $\tau_n$ which is the usual normalized trace. That is, if $X=[x_{kl}]\in\MM_n$, we define $\tau_n(X)=\frac{1}{n}\tr(X)=\frac{1}{n}\sum_{k=1}^nx_{kk}.$ If $\mathcal{A}$ is a finite dimensional $C^*$-algebra, then we have that $\mathcal{A}=\MM_{n_1}\oplus ... \oplus \MM_{n_k},$ for some $n_1,...,n_k\in\NN$. In this case every tracial state $\tau$ on $\mathcal{A}$ has the form \begin{align*}
% \tau(X_1\oplus ... \oplus X_k)=t_1\tau_{n_1}(X_1)+...+t_k\tau_{n_k}(X_{k}),
% \end{align*} where $t_j\geq 0$ such that $t_1+...+t_k=1$(called the {\it weights}), and $\tau_{n_j}$ is the normalized trace on $\MM_{n_j}$ for each $1\leq j\leq k$. 
% These observations and Remark \ref{remark-q-in-qc} lead to the following characterization of $C_q^s(n,m)$:
% 
% \begin{prop} We have that $(p(i,j|v,w)) \in C_q^s(n,m)$ if and only if there exist $k \in \bb N$, $\, n_{\ell} \in \bb N, 1 \le \ell \le k$, weights $t_1,...,t_k$, projections $P_{v,i,\ell} \in \bb M_{n_{\ell}}, 1 \le v \le n, 1 \le i \le m, \, 1 \le \ell \le k$ satisfying $\sum_i P_{v,i,\ell} = I_{n_{\ell}}$ such that
% \[ p(i,j|v,w) = \sum_{\ell=1}^k t_{\ell} \tau_{n_{\ell}}(P_{v,i,\ell}P_{w,j, \ell}).\]
% \end{prop}

This theorem and Remark \ref{remark-q-in-qc} lead to the following characterization of $C_q^s(n,m)$:

\begin{prop}
We have that $(p(i,j|v,w)) \in C_q^s(n,m)$ if and only if there exists a finite dimensional C$^*$-algebra $\mathcal{A}$
with a tracial state $\tau$ and with a family of projections
$\{e_{v,i}:v\in I, i\in O\}\subset \mathcal{A}$ such that $\sum_{i\in O}e_{v,i}=1$ for all $v\in I$ and $p(i,j|v,w)=\tau(e_{v,i}e_{w,j})$ for all $i,j,v,w$.
\end{prop}

The set of synchronous vector correlations is described in the next proposition. 

\begin{prop}\label{vectsyn}
We have $(p(i,j|v,w))\in C_{vect}^s(n,m)$ if and only if \begin{align*}
p(i,j|v,w)=\langle x_{v,i},x_{w,j}\rangle
\end{align*} for a set of vectors $\{x_{v,i}:v\in I,i\in O\}\subset \mathcal{H}$ with $x_{v,i}\perp x_{v,j}$ when $i\neq j$, $\sum_{i=1}^{m}x_{v,i}=h$ for some unit vector $h\in\mathcal{H}$, and $\langle x_{v,i}, x_{w,j}\rangle\geq 0$.
\end{prop}

The synchronous subsets satisfy inclusions as in expression \ref{order2},
\begin{equation*}%\label{eq:synchincl}
C_q^s(n,m)\subseteq C_{qa}^s(n,m)\subseteq C_{qc}^s(n,m)\subseteq C_{vect}^s(n,m)\subseteq \RR^{n^2m^2},
\end{equation*}
and since $C_{qa}(n,m)$, $C_{qc}(n,m)$, and $C_{vect}(n,m)$ are closed sets it is easy to see that their synchronous subsets are also closed. We can also ask the synchronous analogues of the questions described before. It is easy to see that, $C_t(n,m) = C_{t^{\prime}}(n,m) \implies C_t^s(n,m) = C_{t^{\prime}}^s(n,m)$,
but there is no a priori reason that the converses should hold. % $\overline{C_{q}(n,m)}=C_{qa}(n,m)$, does this also happen at the synchronous level? 
It is shown in \cite{dykema2015} that $\overline{C_{q}^s(n,m)}=C_{qc}^s(n,m)$ for all $n,m\in\NN$ is equivalent to Connes' embedding conjecture.  In \cite{KPS, Fritz} it is shown that $\overline{C_q^s(n,m)} = C_{qa}^s(n,m)$.

The questions described above can be formulated in terms of values of games. If we restrict $\omega_{\cl F}(\cl G,\pi)$ to the synchronous subset $\cl F^s$ of $\cl F$, we obtain the \textit{synchronous value} of the game $\cl G$ given the probability density $\pi$ defined by \begin{align*}
\omega_{\cl F}^s(\cl G,\pi) = \sup\{V(p,\pi): p\in \cl F^s\}.
\end{align*} As before we write this as $\omega_t^s(\cl G, \pi)$ when $\cl F= C_t(n,m)$. The following proposition relates  the synchronous values of a game to Connes' embedding conjecture.

\begin{prop}[Proposition 4.1, \cite{dykema2015}]
If Connes' embedding conjecture is true then $\omega_q(\mathcal{G},\pi)=\omega_{qc}(\mathcal{G},\pi)$ and $\omega_q^s(\mathcal{G},\pi)=\omega_{qa}^s(\mathcal{G},\pi)=\omega_{qc}^s(\mathcal{G},\pi)$ hold for every game $\mathcal{G}$ and every distribution $\pi$.
\end{prop}

\begin{remark} It is not known if the converse of any of these above implications is true.  That is, for example, if $\omega_q(\cl G, \pi) = \omega_{qc}(\cl G, \pi), \forall \cl G, \, \forall \pi$, then must Connes' embedding conjecture be true?
\end{remark}

We now introduce the $\Delta$ game.

\section{The $\Delta$ Game}
\label{delta-section}

The $\Delta$ game is a nonlocal game with three inputs and two outputs. We have $I=\{0,1,2\}$  as the input set and $O=\{0,1\}$ as the output set (thus $n=3,m=2$). Out of the 36 possible tuples $(v,w,i,j)$, allowed rules $(v,w,i,j)\in I\times I\times O\times O$ are \begin{align*}
(0,0,0,0),\quad(0,1,0,1),\quad(1,1,0,0),\quad(1,2,0,1),\quad(2,2,0,0),\quad(2,0,0,1), \\
(0,0,1,1),\quad(0,1,1,0),\quad(1,1,1,1),\quad(1,2,1,0),\quad(2,2,1,1),\quad(2,0,1,0),
\end{align*} whereas the disallowed rules are \begin{align*}
(0,0,0,1),\quad(0,1,0,0),\quad(1,1,0,1),\quad(1,2,0,0),\quad(2,2,0,1),\quad(2,0,0,0),\\
(0,0,1,0),\quad(0,1,1,1),\quad(1,1,1,0),\quad(1,2,1,1),\quad(2,2,1,0),\quad(2,0,1,1).
\end{align*} The remaining 12 tuples $(v,w,i,j)$ are all allowed. 

The first 12 allowed rules may be visualized as in Figure \ref{figure1}. The allowed edges $(0,0), (1,1), (2,2)$ are shown with dashed lines while $(0,1), (1,2), (2,0)$ are shown with solid lines. The dashed lines are \textit{even} while the solid lines are \textit{odd}. This means that if Alice and Bob are given inputs joined by dashed lines then they return outputs with even sum; and in the other case they return outputs with odd sum. 

\begin{figure}[h]
\centering
\begin{tikzpicture}
\node at (-1,-0.5) {Alice};
\node at (1,-0.5) {Bob};
\node[circle,fill=black,inner sep=1pt,minimum size=1pt]  at (-1,-1) {};
\node[circle,fill=black,inner sep=1pt,minimum size=1pt]  at (1,-1) {};
\node at (-1.25,-1) {\Large 0};
\node at (1.25,-1) {\Large 0};
\node[circle,fill=black,inner sep=1pt,minimum size=1pt]  at (-1,-2) {};
\node[circle,fill=black,inner sep=1pt,minimum size=1pt]  at (1,-2) {};
\node at (-1.25,-2) {\Large 1};
\node at (1.25,-2) {\Large 1};
\node[circle,fill=black,inner sep=1pt,minimum size=1pt]  at (-1,-3) {};
\node[circle,fill=black,inner sep=1pt,minimum size=1pt]  at (1,-3) {};
\node at (-1.25,-3) {\Large 2};
\node at (1.25,-3) {\Large 2};
\draw[-,thick,dashed] (-1,-1) -- (1,-1);
\draw[-,thick,dashed] (-1,-2) -- (1,-2);
\draw[-,thick,dashed] (-1,-3) -- (1,-3);
\draw[-,thick] (-1,-1) -- (1,-2);
\draw[-,thick] (-1,-2) -- (1,-3);
\draw[-,thick] (-1,-3) -- (1,-1);
\end{tikzpicture}
\caption{$\Delta$ game rule function.}
\label{figure1}
\end{figure}
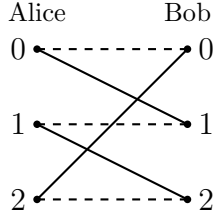

Alice and Bob receive inputs according to the uniform distribution $\pi=(\pi(v,w))$ on the set of inputs \begin{align*}
E=\{(0,0), (1,1), (2,2), (0,1), (1,2), (2,0)\},
\end{align*} that is, $\pi(v,w)=\frac{1}{6}$ for all $(v,w)\in E$ (and zero otherwise). To compute the synchronous value of the game given the distribution $\pi$ we first compute the value of a single correlation $p=(p(i,j|v,w))$, which is, \begin{align*}
V(p,\pi) = \frac{1}{6}\left( \sum_{v=0}^2\sum_{i=0}^1p(i,i|v,v)+p(i,i+1|v,v+1) \right),
\end{align*} so that the value of the game becomes, \begin{align*}
\omega_t^s(\mathcal{G},\pi)=\sup\left\lbrace \frac{1}{6}\left( \sum_{v=0}^2\sum_{i=0}^1 p(i,i|v,v)+p(i,i+1|v,v+1) \right):p(i,j|v,w)\in C_t^s(3,2)\right\rbrace,
\end{align*} where $t\in\{q, qa, qc, vect\}$. Denote the expression inside the braces by,
\begin{align*}
\widetilde{\theta}= \frac{1}{6}\left( \sum_{v=0}^2\sum_{i=0}^1 p(i,i|v,v)+p(i,i+1|v,v+1) \right).
\end{align*} We will use Theorem \ref{qcsyn} and Proposition \ref{vectsyn} to simplify $\widetilde{\theta}$ and to obtain expressions involving operators and vectors in the case of $t=qc$ and $t=vect$, respectively. Moreover, when $t=q$, by Remark \ref{remark-q-in-qc} it suffices to proceed as in the case $t=qc$ using Theorem \ref{qcsyn} to simplify $\widetilde{\theta}$, but restricting to the case of operators on finite dimensional Hilbert spaces.

%\marginpar{We need letters other than $\theta$ and $\tilde{\theta}$ throughout. It is too confusing to %be have $\theta$ denote the value of the game and a variable!!!}

We first handle the $t=qc$ case. By Theorem \ref{qcsyn}, a correlation $(p(i,j|v,w))$ is in $C_{qc}^s(3,2)$ if and only if there exists a C$^*$-algebra $\mathcal{A}$ of $\bh$ generated by a family of projections $\{A_{v,i}:i=0,1 \text{ and }v=0,1,2\}$ satisfying $A_{v,0}+A_{v,1}=I_{\mathcal{H}}$ for $v\in\{0,1,2\}$ and a tracial state $\tau:\mathcal{A}\rightarrow\CC$ such that $p(i,j|v,w)=\tau(A_{v,i}A_{w,j})=\left\langle (A_{v,i}A_{w,j})h,h\right\rangle$, for some unit vector $h\in\mathcal{H}$. For notational convenience we define \begin{align*}
A_0=A_{0,0}, \quad\quad A_1=A_{1,0}, \quad\quad A_2=A_{2,0}.
\end{align*} Then $A_{v,1}=I_{\mathcal{H}}-A_v=I_{\mathcal{H}}-A_{v,0}$ for $v\in\{0,1,2\}$. Using this we can rewrite $\widetilde{\theta}$ as \begin{align*}
\widetilde{\theta} &= \frac{1}{6}\sum_{v=0}^2\sum_{i=0}^1 p(i,i|v,v)+p(i,i+1|v,v+1) \\
&= \frac{1}{6}\sum_{v=0}^2\sum_{i=0}^1 \tau(A_{v,i}A_{v,i})+\tau(A_{v,i}A_{v+1,i+1}) \\
&= \frac{1}{2}+\frac{1}{3}\tau(A_0+A_1+A_2)-\frac{1}{3}\sum_{v=0}^2\tau(A_vA_{v+1}).
\end{align*} We now define a ``parameter'' $\theta$ by setting
\begin{align*}
\theta &= \frac{1}{3}\tau(A_0+A_1+A_2),
\end{align*} which enables us to write $\widetilde{\theta}$ as \begin{align}\label{eq-in-qc-case}
\widetilde{\theta} = \frac{1}{2}+\frac{1}{3}\tau(A_0+A_1+A_2)-\frac{1}{3}\sum_{v=0}^2\tau(A_vA_{v+1}) 
= \frac{1}{2}+\theta-\frac{1}{3}\sum_{v=0}^2\tau(A_vA_{v+1}).
\end{align}

Similarly, in the $t=vect$ case, using Proposition \ref{vectsyn} and proceeding as in the previous paragraph,
writing $x_i$ for $x_{i,0}$,
we see that $\widetilde{\theta}$ is given by \begin{align*}
\widetilde{\theta} = \frac{1}{2} + \frac{1}{3}\left\langle x_0+x_1+x_2,h \right\rangle -\frac{1}{3} \sum_{v=0}^2\langle x_v,x_{v+1}\rangle, 
\end{align*} for some set of vectors $\{x_0,x_1,x_2,h\}$ in some Hilbert space $\cl H$
satisfying $\|h\|=1$ and, for all $v$ and $w$,
\[
x_v\perp(h-x_v),\qquad\langle x_v,x_w\rangle\ge0,\qquad\langle x_v,h-x_w\rangle\ge0,\qquad\langle h-x_v,h-x_w\rangle\ge0.
\]
Again letting $\theta = \frac{1}{3}\left\langle x_0+x_1+x_2,h \right\rangle$, we may write
\begin{align}\label{eq-in-vect-case}
\widetilde{\theta}= \frac{1}{2}+\theta-\frac{1}{3} \sum_{v=0}^2\langle x_v,x_{v+1}\rangle.
\end{align}

For each $t\in\{q, qa, qc, vect\}$, let $\Theta_t^s$ denote the set of all points $(\theta,\widetilde{\theta})\in\RR^2$ that can be obtained from correlations $(p(i,j|v,w))\in C_t^s(n,m)$ in the manner described above. We want to see how $\Theta_t^s$ behaves under different values of $t$.
It is easy to verify that $\Theta_t^s$ is a convex set since it is the affine image of the convex set $C_t^s(n,m)$. To find $\Theta_t^s$, it is enough to compute the following two functions for each $\theta$, \begin{align*}
f_t^u(\theta) = \sup\{\widetilde{\theta}:(\theta,\widetilde{\theta})\in\Theta_t^s\}, \qquad f_t^l(\theta) = \inf\{\widetilde{\theta}:(\theta,\widetilde{\theta})\in\Theta_t^s\},
\end{align*} where \textit{u} and \textit{l} stand for \textit{upper} and \textit{lower}, respectively. We also need to determine if the supremum and the infimum are attained or not. Notice that in the $qc$ case, in order to find the supremum (resp., infimum) of $\widetilde{\theta}=\frac{1}{2}+\theta-\frac{1}{3}\sum_{v=0}^2\tau(A_vA_{v+1}),$ we need to find the infimum (resp., supremum) of the quantity $\sum_{v=0}^2\tau(A_vA_{v+1})$. A similar statement holds for the $vect$ case.

In the $qc$ case, notice that since $A_v$'s are projections and $\tau$ is a state we get, $0\leq  \frac{1}{3} \tau(A_0+A_1+A_2) \leq 1.$ Similarly in the $vect$ case, by the Cauchy-Schwarz inequality we get $0\leq \frac{1}{3}\langle x_0+x_1+x_2,h\rangle\leq 1$. Hence $0\leq \theta\leq 1$. Conversely, if $\theta\in [0,1]$, then we can always find projections $A_0,A_1,A_2$ in some C$^*$-algebra with a tracial state $\tau$, such that $\frac{1}{3}\tau(A_0+A_1+A_2)=\theta$.

%\begin{Lem}\label{proj-for-t}
%Let $t\in [0,1]$, then there exists a projection $P$ in some finite dimensional $C^*$-algebra $\mathcal{A}$ such that $\tau(P)=t$ for some tracial state $\tau$ on $\mathcal{A}$.
%\end{Lem}
%\begin{proof}
%First suppose that $t$ is rational with $t=\frac{p}{q}$ ($0<p<q$). Consider the $q\times q$ matrix $P=\begin{bmatrix}
%I_p & 0 \\ 0 & 0_{q-p}
%\end{bmatrix}$, then $\tau(P)=\frac{p}{q}=t$. Next let $0<t<1$ be irrational. Let $\mathcal{A}=\MM_n\oplus \MM_n$ and $P=I_n\oplus 0$. Define a tracial state on $\mathcal{A}$ by $\tau(X_1\oplus X_2)=t\tau(X_1)+(1-t)\tau(X_2)$. Then clearly $\tau(P)=t$.
%\end{proof} 

It is evident that $\Theta_q^s \subseteq \Theta_{qa}^s \subseteq \Theta_{qc}^s \subseteq \Theta_{vect}^s$.

\begin{theorem}\label{mainth}
For $t \in \{ q, qa, qc \}$, we have
\begin{equation}\label{eq:flutmain}
f_t^l(\theta) = \frac{1}{2}, \qquad f_t^u(\theta) = \begin{cases} 
\hfill \frac{1}{2}+\theta \hfill & \text{ for } 0\leq \theta \leq \frac{1}{3} \\
\hfill \frac{3+\theta}{4} \hfill & \text{ for } \frac{1}{3}\leq \theta\leq \frac{1}{2} \\
\hfill \frac{4-\theta}{4} \hfill & \text{ for } \frac{1}{2} \leq \theta \leq \frac{2}{3} \\
\hfill \frac{3}{2}-\theta \hfill & \text{ for } \frac{2}{3}\leq \theta \leq 1.
\end{cases}
\end{equation}
Moreover, we have
\begin{equation}\label{eq:fluvectmain}
f_{vect}^l(\theta) = \frac{1}{2}, \qquad f_{vect}^u(\theta) = \begin{cases} 
\hfill \frac{1}{2}+\theta \hfill & \text{ for } 0\leq \theta \leq \frac{1}{3} \\
\hfill \frac{1+3\theta-3\theta^2}{2} \hfill & \text{ for } \frac{1}{3}\leq \theta \leq \frac{2}{3} \\
\hfill \frac{3}{2}-\theta \hfill & \text{ for } \frac{2}{3}\leq \theta \leq 1.
\end{cases}
\end{equation}
In all of the these cases,  the infimum and supremum are attained by both $f_t^u$ and $f_t^l$.
Since $(\theta,\widetilde{\theta})\in\Theta_t^s$ if and only if $0\le\theta\le1$ and $f_t^l(\theta)\leq \widetilde{\theta}\leq f_t^u(\theta)$,
we see that $\Theta_t^s$ is a closed set in $\RR^2$ for each $t\in \{q,qa,qc,vect\}$.
In particular, we have 
\begin{equation}\label{eq:Thetaincl}
\Theta_q^s = \Theta_{qa}^s = \Theta_{qc}^s \subsetneq \Theta_{vect}^s.
\end{equation}
\end{theorem}

The functions as obtained in Theorem \ref{mainth} are shown in Figure~\ref{fig:fplots}.
%\marginpar{The graph disappears when I tex the file. This means that standard tex packages can not produce your picture--this will be another problem when we submit}
\iffalse \begin{figure}[h]
\centering
\begin{tikzpicture}[xscale=8,yscale=8]
\draw[->] (0,0.45) -- (1.2,0.45) node[right] {$\theta$};
\draw[->] (0,0.45) -- (0,1.2) node[above] {$\widetilde{\theta}$};
\draw[domain=0:(1/3),smooth,variable=\x,blue] plot ({\x},{(1/2)+\x});
\draw[domain=(1/3):(1/2),smooth,variable=\x,blue] plot ({\x},{(3+\x)/4});
\draw[domain=(1/2):(2/3),smooth,variable=\x,blue] plot ({\x},{(4-\x)/4});
\draw[domain=(2/3):1,smooth,variable=\x,blue] plot ({\x},{(3/2)-\x});
\draw[domain=0:1,smooth,variable=\x,blue] plot ({\x},{1/2});

\draw[domain=0:(1/3),smooth,variable=\y,red] plot ({\y},{(1/2)+\y});
\draw[domain=(1/3):(2/3),smooth,variable=\y,red] plot ({\y},{(1+3*\y-3*\y*\y)/2});
\draw[domain=(2/3):1,smooth,variable=\y,red] plot ({\y},{(3/2)-\y});
\end{tikzpicture}
\caption{Plots of $f^l_t=f^l_{vect}$, $f^u_t$ and $f^u_{vect}$ from Theorem~\ref{mainth}}
\label{fig:fplots}
\end{figure}\fi
% Note that the final proper containment in~\eqref{eq:Thetaincl}
% provides a proof that $C_{qc}(3,2)$ is a proper subset of $C_{vect}(3,2)$, thus providing another
% counterexample to the strong Tsirelson conjecture.

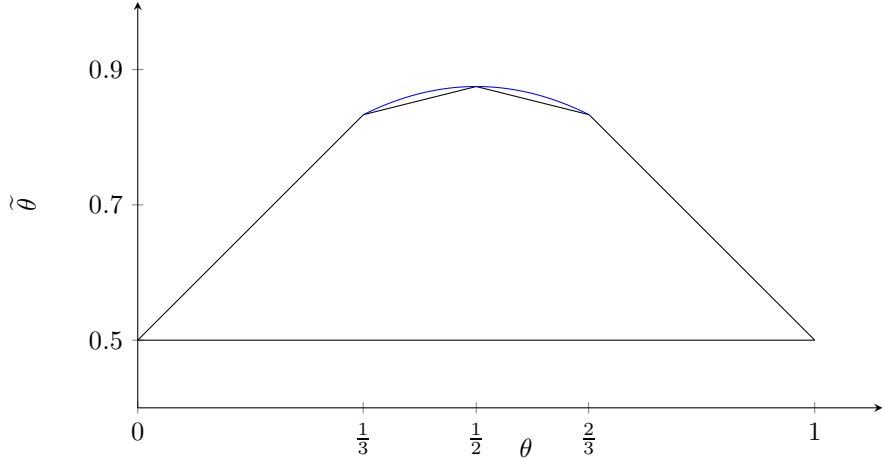
\begin{figure}[h]
\centering
\begin{tikzpicture}
\begin{axis}[
	x=9cm, y=9cm,
    axis lines = left,
	xlabel style={right},
	ylabel style={above},
	xmin=0, xmax=1.1,
    xlabel = $\theta$,
    xtick={0,0.333,0.5,0.666,1},
    xticklabels={0,$\frac{1}{3}$,$\frac{1}{2}$,$\frac{2}{3}$,1},
	ymin=0.4, ymax=1,    
    ylabel = {$\widetilde{\theta}$},
    ytick={0.5,0.7,0.9},
    yticklabels={0.5,0.7,0.9},
]

\addplot [
    domain=0:1/3, 
    samples=100, 
]
{(1/2)+x};

\addplot[
	domain=1/3:1/2,
	samples=100,
]
{(3+x)/4};

\addplot[
	domain=1/2:2/3,
	samples=100,
]
{(4-x)/4};

\addplot[
	domain=2/3:1,
	samples=100,
]
{(3/2)-x};

\addplot[
	color=blue,
	domain=1/3:2/3,
	samples=100,
]
{(1+3*x-3*x^2)/2};

\addplot[
	domain=0:1,
	samples=100,
]
{1/2}; 
\end{axis}
\end{tikzpicture}
\caption{Plots of $f^l_t=f^l_{vect}$, $f^u_t$ and $f^u_{vect}$ from Theorem~\ref{mainth}}
\label{fig:fplots}
\end{figure}

The fact that the functions $f^u_{vect}$ and $f^u_{qc}$ are different allows us to deduce the following.

\begin{Co} We have that $C^s_{qc}(3,2) \subsetneq C^s_{vect}(3,2)$.
\end{Co}

\begin{remark}  There is another larger set of correlations that we could have considered, the {\em nonsignalling correlations}. For a definition, see \cite{LR}. If we let  $C^s_{ns}(n,k)$ denote the set of synchronous nonsignalling correlations, then it is shown in \cite{LR} that the set $C^s_{ns}(n,2)$ is a polytope. If we let $f^u_{ns}$ denote the analogous function obtained by taking the supremum over the set of synchronous nonsignalling correlations, then the fact that the set of such correlations is a polytope implies that $f^u_{ns}$ would be piecewise linear. Hence,  $f^u_{vect} \ne f^u_{ns}$ and we can conclude that $C^s_{vect}(3,2) \subsetneq C^s_{ns}(3,2)$.
\end{remark}

\section{The case of  $t=vect$.}

In this section, we compute $f^l_{vect}$ and $f^u_{vect}$ to prove~\eqref{eq:fluvectmain} in Theorem~\ref{mainth}.
We will employ the symmetrization provided by the next lemma.

\begin{lemma}\label{lem:symm}
$\Theta_{vect}^s$ is equal to the set of pairs $(\theta,\thetat)$  with $0\le \theta\le 1$,
such that there exist vectors $x_0,x_1,x_2,h$
in a Hilbert space with the properties:
\begin{itemize}\renewcommand{\labelitemi}{$\bullet$}
\item $\|h\|=1$,
\item $\forall v$ $\langle x_v,h\rangle=\langle x_v,x_v\rangle=\theta$,
\item $\forall v$ $\langle x_v,x_{v+1}\rangle=\beta$, where $\thetat=\frac12+\theta-\beta$ and $2\theta-1\le\beta\le\theta$.
\end{itemize}
\end{lemma}
\begin{proof}
By Proposition 2.4, and the discussion in Section~\ref{delta-section},
$\Theta_{vect}^s$ is the set of pairs $(\theta,\thetat)$ such that there exist vectors $x_0,x_1,x_2,h$ in a Hilbert space $\Hc$
with the properties that $\|h\|=1$,
for all $v$ and $w$, we have
\[
\langle x_v,h\rangle=\langle x_v,x_v\rangle,\quad\langle x_v,x_w\rangle\ge0,\quad\langle x_v,h-x_w\rangle\ge0,\quad\langle h-x_v,h-x_w\rangle\ge0
\]
and, moreover,
\[
\frac13\sum_{v=0}^2\langle x_v,h\rangle=\theta,\qquad
\frac13\sum_{v=0}^2\langle x_v,x_{v+1}\rangle=\beta,
\]
where $\thetat=\frac12+\theta-\beta$.
The conditions appearing in the lemma are precisely these, but with the additional requirement that the quantities $\langle x_v,h\rangle$
and $\langle x_v,x_{v+1}\rangle$ are the same for all $v\in\{0,1,2\}$. 
However, given $x_0,x_1,x_2,h$ satisfying these weaker conditions and considering
\[
\hti=\frac1{\sqrt3}(h\oplus h\oplus h),\qquad \xt_v=\frac1{\sqrt3}(x_v\oplus x_{v+1}\oplus x_{v+2})
\]
in the Hilbert space $\Hc^{\oplus3}$,
we see that $\xt_0,\xt_1,\xt_2,\hti$ satisfy the stronger conditions and yield the same pair $(\theta,\thetat)$.
\end{proof}

We now prove the part of Theorem \ref{mainth} involving the case $t=vect$.

\begin{theorem}\label{mainthm-vect}
The functions
\begin{equation}\label{eq:fluvectdef}
f_{vect}^l(\theta)= \inf\{\widetilde{\theta}:(\theta,\widetilde{\theta})\in\Theta_{vect}^s\},\qquad
f_{vect}^u(\theta)=\sup\{\widetilde{\theta}:(\theta,\widetilde{\theta})\in\Theta_{vect}^s\}
\end{equation}
are given by
\begin{equation}\label{eq:fluvect}
f_{vect}^l(\theta) = \frac{1}{2}, \qquad f_{vect}^u(\theta) = \begin{cases} 
\hfill \frac{1}{2}+\theta \hfill & \text{ for } 0\leq \theta \leq \frac{1}{3} \\
\hfill \frac{1+3\theta-3\theta^2}{2} \hfill & \text{ for } \frac{1}{3}\leq \theta \leq \frac{2}{3} \\
\hfill \frac{3}{2}-\theta \hfill & \text{ for } \frac{2}{3}\leq \theta \leq 1.
\end{cases}
\end{equation}
Moreover, both the infimum and supremum are attained, for all values of $\theta\in[0,1]$.
\end{theorem}

\begin{proof}
Fix $\theta\in[0,1]$.
By Lemma~\ref{lem:symm}, we are interested in the set of $\beta$ such that there exist vectors $x_0,x_1,x_2,h$ in some Hilbert space
satisfying the conditions listed there.
Let $y_v=h-x_v$.
Consider the Gramian matrix $G$ associated with the
seven vectors $h,x_0,x_1,x_2,y_0,y_1,y_2$.
The conditions of Lemma~\ref{lem:symm} imply that this is the $7\times 7$ matrix
\[
G= \begin{bmatrix} 1& \theta& \theta& \theta & 1-\theta & 1-\theta& 1-\theta\\
\theta& \theta & \beta & \beta & 0 & \theta-\beta & \theta-\beta \\
\theta& \beta & \theta & \beta & \theta-\beta & 0 & \theta-\beta \\
\theta & \beta & \beta & \theta & \theta-\beta & \theta-\beta & 0 \\
1-\theta & 0 & \theta-\beta & \theta-\beta & 1-\theta & 1+\beta -2\theta & 1+\beta -2\theta\\
1-\theta & \theta-\beta & 0 & \theta-\beta & 1+\beta -2\theta & 1-\theta & 1+\beta -2\theta\\
1-\theta & \theta-\beta & \theta-\beta & 0 & 1+\beta -2\theta & 1+\beta -2\theta & 1-\theta \end{bmatrix}
\]
and furthermore, that $G$ is positive semidefinite and 
\begin{equation}\label{eq:thetabeta}
\max(0,2\theta-1)\le\beta\le\theta.
\end{equation}
Conversely, given any such  $7\times 7$ positive semidefinite matrix
and with the additional condition~\eqref{eq:thetabeta}, 
we can construct seven such vectors in a Hilbert space.
Thus, we are interested in the set of $\beta$ that satisfy~\eqref{eq:thetabeta} and yield a positive semidefinite matrix $G$ given above.

We apply one step of the Cholesky algorithm, and conclude that the $7 \times 7$ matrix $G$ is positive semidefinite if and only if the following $6 \times 6$ matrix $G'$ is positive semidefinite:
\[
G'=\begin{bmatrix}  \theta-\theta^2 & \beta -\theta^2 & \beta-\theta^2 & \theta^2 -\theta & \theta^2 -\beta & \theta^2 -\beta \\
\beta- \theta^2 & \theta - \theta^2 & \beta-\theta^2 & \theta^2 -\beta & \theta^2 -\theta & \theta^2 -\beta\\
\beta-\theta^2 & \beta - \theta^2 & \theta -\theta^2 & \theta^2 -\beta & \theta^2 -\beta & \theta^2 -\theta\\
\theta^2 -\theta & \theta^2 -\beta  & \theta^2 -\beta & \theta-\theta^2 & \beta -\theta^2 & \beta -\theta^2\\
\theta^2 -\beta & \theta- \theta^2 & \theta^2 -\beta & \beta-\theta^2 & \theta-\theta^2 & \beta -\theta^2 \\
\theta^2 -\beta & \theta^2 -\beta & \theta^2 -\theta & \beta -\theta^2 & \beta -\theta^2 & \theta -\theta^2 \end{bmatrix}.
\]
This matrix $G'$ partitions into a block matrix of the form 
$\begin{bmatrix} A & -A\\ -A & A \end{bmatrix},$ where \begin{align*}
A = \begin{bmatrix} a & x & x\\x & a & x\\ x & x & a \end{bmatrix},
\end{align*} with $a= \theta - \theta^2$ and $x = \beta - \theta^2$.
Thus the matrix $G'$ is positive semi-definite if and only if $A\geq 0$.
Using the determinant criteria we see that $A \ge 0$ if and only if $|x| \le a$ and $2x^3 - 3ax^2 + a^3 \ge 0$.
Simplifying we see that $A\geq 0$ if and only if $-\frac{a}{2} \le x \le a$.
Substituting the values of $a$ and $x$, we find that the Gramian matrix $G$ is positive semidefinite if and only if
\[
\frac{ 3 \theta^2 -\theta}{2} \le \beta \le \theta.
\]
Thus, the set of all possible $\beta$ is the set satisfying
\[
\max \left\lbrace \frac{3\theta^2 -\theta}{2}, 2\theta-1, 0 \right\rbrace \le \beta \le \theta.
\]
This becomes
\[
\begin{aligned}
\hfill 0 \leq \beta \leq \theta &\quad\text{for }\textstyle0\leq \theta \leq \frac{1}{3} \\
\textstyle\hfill \frac{3\theta^2-\theta}{2} \leq \beta \leq \theta &\quad\text{for }\textstyle\frac{1}{3} \leq \theta \leq \frac{2}{3} \\
\hfill 2\theta-1 \leq \beta \leq \theta&\quad\text{for }\textstyle\frac{2}{3} \leq \theta \leq 1.
\end{aligned}
\]
Thus, we obtain the values~\eqref{eq:fluvect} and we have that the infimum and supremum in~\eqref{eq:fluvectdef} are attained.
\end{proof}

\section{The cases $t\in\{q,qa,qc\}$.}

In this section, we compute $f^l_t$ and $f^u_t$ when $t\in\{q,qa,qc\}$ to prove~\eqref{eq:flutmain} in Theorem~\ref{mainth}.
We begin with a symmetrization lemma, analogous to Lemma~\ref{lem:symm}

\begin{lemma}\label{lem:symmC*}
The set $\Theta_{qc}^s$ (resp., $\Theta_q^s$), is equal to the set of pairs $(\theta,\thetat)$  with $0\le \theta\le 1$,
such that there exists a C$^*$-algebra $\mathcal{A}$ (resp., a finite dimensional C$^*$-algebra, $\Ac$)
with a faithful tracial state $\tau$ and with projections $A_0,A_1,A_2\in\Ac$ such that for all $v$, 
\begin{equation}\label{eq:tauAv}
\tau(A_v)=\theta,\qquad\tau(A_vA_{v+1})=\beta,
\end{equation}
where  $\thetat=\frac12+\theta-\beta$.
\end{lemma}
\begin{proof}
By Theorem~\ref{qcsyn} and the discussion in Section~\ref{delta-section}, $(\theta,\thetat)$ belongs to $\Theta_{qc}^s$ (respectively, $\Theta_q^s$) if and only if there is a C$^*$-algebra $\Ac$ (respectively, a finite dimensional C$^*$-algebra $\Ac$), with a faithful tracial state $\tau$ and projections
$A_0,A_1,A_2$ such that 
\[
\frac13\sum_{v=0}^2\tau(A_v)=\theta,\qquad\frac13\sum_{v=0}^2\tau(A_vA_{v+1})=\beta,
\]
where $\thetat=\frac12+\theta-\beta$.
But if such exist, then we can consider the C$^*$-algebra $\widetilde{\Ac}=\Ac\oplus\Ac\oplus\Ac$ with the trace $\widetilde{\tau}=\frac{1}{3}\tau \oplus \frac{1}{3}\tau \oplus \frac{1}{3}\tau$, and projections $\widetilde{A}_v=A_v\oplus A_{v+1}\oplus A_{v+2}$ that satisfy the stronger requirements of the lemma that include~\eqref{eq:tauAv}.
\end{proof}

We now have some C$^*$-algebra results.

\begin{prop}\label{central-lemma}
Let $\mathcal{A}$ be a unital $C^*$-algebra with a faithful tracial state $\tau$.
Let $A$ and $P$ be hermitian elements in $\mathcal A$.
If $AP-PA\neq 0$, then there exists $H=H^*\in \mathcal{A}$ such that, letting $f(t)=\tau(A(e^{iHt}Pe^{-iHt}))$ for $t\in \RR$ , we have $f'(0)> 0$.
\end{prop}
\begin{proof}
If $H\in \mathcal A$ is hermitian, then
\begin{align*}
f'(0) =  i\tau(AHP-APH) =i\tau((PA-AP)H),
\end{align*} where we used the fact that $\tau$ is a tracial state.
Supppose $AP-PA\neq 0$.
Let $H=i(PA-AP)$.
Then $H$ is hermitian and $f'(0)=\tau(|PA-AP|^2)>0$, where the strict inequality follows beacuse $AP-PA\neq 0$ and $\tau$ is a faithful state.
\end{proof}

% We don't actually use the following corollary, but its proof is an easier verion of the proof of Corollary~\ref{three-projections-lemma},
% which we will use.
 
% \begin{Co}\label{commuting-lemma}
% Let $\mathcal{A}$ be a unital $C^*$-algebra with a faithful tracial state $\tau$. Fix $A\geq 0\in \mathcal{A}$ and $\lambda\geq 0$. Define $m=\inf\{\tau(AP):P\geq 0, \tau(P)=\lambda\}$. If there exists $P\geq 0$ in $\mathcal{A}$ with $\tau(P)=\lambda$ such that $\tau(AP)=m$, then $AP=PA$.
% \end{Co}
 
% \begin{proof}
% We proceed by contrapositive. Suppose $AP\neq PA$. Then by Proposition \ref{central-lemma}, there exists $H=H^*$ in $\mathcal{A}$ such that if $\Gamma(t)=A(e^{iHt}Pe^{-iHt})$ and $f(t)=\tau(\Gamma(t))$, then $f'(0)>0$. If we let $P_t=e^{iHt}Pe^{-iHt}$, then clearly $\tau(P_t)=\lambda$.  Since $f'(0)>0$, taking a sufficiently small negative $t$, we get $f(t)<f(0)$ which means $\tau(AP_t)<\tau(AP_0)=\tau(AP)=m$, which implies that $m$ is not the infimum.
% \end{proof}

\begin{Co}\label{three-projections-lemma}
Let $\mathcal{A}$ be a unital $C^*$-algebra with a faithful tracial state $\tau$. Fix $\theta\in [0,1]$.
Let
\begin{multline*}
\beta = \inf \bigg\{\frac{1}{3}\tau\left(AB+BC+CA \right):A, B, C \in \cl A \text{ projections}, \\
\tau(A)=\tau(B)=\tau(C)=\theta \bigg\}.
\end{multline*}
If there exist projections $A_0, B_0, C_0$ in $\mathcal{A}$ such that $\tau(A_0)=\tau(B_0)=\tau(C_0)=\theta$
and $\beta=\frac{1}{3}\tau(A_0B_0+B_0C_0+C_0A_0)$, then
\[
[A_0,B_0+C_0]=[B_0,C_0+A_0]=[C_0,A_0+B_0]=0.
\]
\end{Co}
\begin{proof}
We will show that $A_0$ commutes with $B_0+C_0$ and the other commutation relations follow by symmetry.
Let $P=B_0+C_0$.
Suppose, for contradiction, that $[A_0,P]\neq 0$.
Then, by Proposition~\ref{central-lemma},
there exists $H=H^*\in \mathcal{A}$ such that if  $f(t)=\tau(A_0(e^{iHt}Pe^{-iHt}))$, then $f'(0)> 0$.
Fix some small and negative $t$ such that $f(t)<f(0)$.
Letting
$B_t=e^{iHt}B_0e^{-iHt}$ and $C_t=e^{iHt}C_0e^{-iHt}$, we see that $B_t$ and $C_t$ are themselves projections in $\mathcal{A}$ and $\tau(B_t)=\tau(C_t)=\theta$.
But then for our value of $t$,
\begin{align*}
\tau(A_0B_t+B_tC_t+C_tA_0) &= \tau(A_0(B_t+C_t)+B_tC_t) \\
&= \tau(A_0(e^{iHt}Pe^{-iHt}))+\tau((e^{iHt}B_0e^{-iHt})(e^{iHt}C_0e^{-iHt})) \\
&= f(t)+\tau(B_0C_0) \\
&< f(0)+\tau(B_0C_0)  = 3\beta,
\end{align*}
which implies that $\beta$ is not the infimum, contrary to hypothesis.
Thus, $A_0$ commutes with $B_0+C_0$.
\end{proof}

We now consider the universal unital C$^*$-algebra $\Afr$ generated by self-adjoint projections $A$, $B$, and $C$ satisfying the commutator relations
\begin{equation}\label{eq:commrels}
[A,B+C]=[B,A+C]=[C,A+B]=0.
\end{equation}
By definition, this is obtained by separation and completion of the universal unital complex algebra generated by noncommuting variables $A$, $B$ and $C$,
under the seminorm that is the supremum of seminorms obtained from Hilbert space representations whereby $A$, $B$ and $C$ are sent
to self-adjoint projections satisfying the above relations.

\begin{prop}\label{prop:univC*}
The universal C$^*$-algebra $\Afr$ described above
is isomorphic to $\CC^8\oplus\MM_2$, wherein
\begin{align*}
A&=0\oplus 0\oplus0\oplus0\oplus1\oplus1\oplus1\oplus1\oplus\left(\begin{matrix}1&0\\0&0\end{matrix}\right), \\
B&=0\oplus 0\oplus1\oplus1\oplus0\oplus0\oplus1\oplus1
 \oplus\left(\begin{matrix}\frac14&\frac{\sqrt3}4\\[1ex]\frac{\sqrt3}4&\frac34\end{matrix}\right), \\
C&=0\oplus 1\oplus0\oplus1\oplus0\oplus1\oplus0\oplus1
 \oplus\left(\begin{matrix}\frac14&-\frac{\sqrt3}4\\[1ex]-\frac{\sqrt3}4&\frac34\end{matrix}\right).
\end{align*}
\end{prop}
\begin{proof}
We will describe all irreducible $*$-representations of $\Afr$ on Hilbert spaces.
Let
\[
Y=2(B+C)-(B+C)^2\in\Afr.
\]
By the commutation relations~\eqref{eq:commrels}, $Y$ commutes with $A$.
We also note that $Y=B+C-BC-CB$ and
\[
BY=B-BCB=YB,
\]
namely, that $Y$ commutes with $B$.
Similarly, $Y$ commutes with $C$.
Hence $Y$ lies in the center of $\Afr$.
Thus, under any irreducible $*$-representation $\pi$, $Y$ must be sent to a scalar multiple of the identity operator.
In other words, we have
\[
\pi(B+C-BC-CB)=\pi(Y)=\lambda\pi(1)
\]
for some $\lambda\in \CC$, so that
\[
\pi(CB)\in\lspan\pi\big(\{1,B,C,BC\}\big).
\]
Similarly, we have
\[
\pi(CA)\in\lspan\pi\big(\{1,A,C,AC\}\big),\qquad
\pi(BA)\in\lspan\pi\big(\{1,A,B,AB\}\big).
\]
Since $\Afr$ is densely spanned by the set of all words in the idempotents $A$, $B$ and $C$, we see
\[
\pi(\Afr)=\lspan\pi\big(\{1,A,B,C,AB,AC,BC,ABC\}\big).
\]
This implies that $\dim\pi(\Afr)\le 8$.
Since $\pi(\Afr)$ is finite dimensional and acts irreducibly on a Hilbert space $\Hc_\pi$, it must be equal to a full matrix algebra.
Considering dimensions, we must have $\dim\Hc_\pi\le 2$.

The irreducible representations $\pi$ of $\Afr$ for which $\dim\Hc_\pi=1$ are easy to describe.
They are the eight representations that send $A$, $B$ and $C$ variously to $0$ and $1$.
We will now characterize the irreducible representations $\pi$ of $\Afr$ for which $\dim\Hc_\pi=2$, up to unitary equivalence.
Let $\pi$ be such a representation.
From the commutation relations~\eqref{eq:commrels}, we see that, if $\pi(A)$ and $\pi(B)$ commute, then also $\pi(C)$ commutes with $\pi(A)$ and with $\pi(B)$,
and the entire algebra $\pi(\Afr)$ is commutative.
This would require $\dim\Hc_\pi=1$.
By symmetry we conclude that no two of $\pi(A)$, $\pi(B)$ and $\pi(C)$ can commute.
In particular, each must be a projection of rank $1$.
After conjugation with a unitary, we must have
\[
\pi(A)=\left(\begin{matrix}1&0\\0&0\end{matrix}\right),
\qquad\pi(B)=\left(\begin{matrix}t&\sqrt{t(1-t)}\\\sqrt{t(1-t)}&1-t\end{matrix}\right)
\]
for some $0<t<1$.
Since $\pi(B)+\pi(C)$ must commute with $\pi(A)$, we must have
\[
\pi(C)=\left(\begin{matrix}c_{11}&-\sqrt{t(1-t)}\\-\sqrt{t(1-t)}&c_{22}\end{matrix}\right),
\]
for some $c_{11},c_{22}\ge0$.
Since $\pi(C)$ is a projection, the only possible choices are (i) $c_{11}=t$ and $c_{22}=1-t$ and
(ii) $c_{11}=1-t$ and $c_{22}=t$.
But in Case (ii), we have $\pi(C)=I_{\Hc_\pi}-\pi(B)$, which violates the prohibition against $\pi(C)$ and $\pi(B)$ commuting.
Thus, we must have
\[
\pi(C)=\left(\begin{matrix}t&-\sqrt{t(1-t)}\\-\sqrt{t(1-t)}&1-t\end{matrix}\right).
\]
Now, using that $\pi(A)+\pi(B)$ and $\pi(C)$ commute, we see that we must have $t=\frac14$ and we easily check that this does provide an
irreducible representation of $\Afr$.

To summarize, up to unitary equivalence, there are exactly nine different irreducible representations of $\Afr$, one of them is two-dimensional and the others
are one-dimensional.
Thus, $\Afr$ is finite dimensional and is isomorphic to the direct sum of the images of its irreducible representations, namely
to $\CC^8\oplus\MM_2$,
with $A$, $B$ and $C$ as indicated.
\end{proof}

We now prove Theorem \ref{mainth} for the cases $t\in\{q,qa,qc\}$.

\begin{theorem}\label{mainthm-qc}
For $t\in\{q,qa,qc\}$, the functions
\begin{equation}\label{eq:fluvectdef-qc}
f_t^l(\theta)= \inf\{\widetilde{\theta}:(\theta,\widetilde{\theta})\in\Theta_t^s\},\qquad
f_t^u(\theta)=\sup\{\widetilde{\theta}:(\theta,\widetilde{\theta})\in\Theta_t^s\}
\end{equation}
are given by
\begin{equation}\label{eq:flutmain-qc}
f_t^l(\theta) = \frac{1}{2}, \qquad f_t^u(\theta) = \begin{cases} 
\hfill \frac{1}{2}+\theta \hfill & \text{ for } 0\leq \theta \leq \frac{1}{3} \\
\hfill \frac{3+\theta}{4} \hfill & \text{ for } \frac{1}{3}\leq \theta\leq \frac{1}{2} \\
\hfill \frac{4-\theta}{4} \hfill & \text{ for } \frac{1}{2} \leq \theta \leq \frac{2}{3} \\
\hfill \frac{3}{2}-\theta \hfill & \text{ for } \frac{2}{3}\leq \theta \leq 1.
\end{cases}
\end{equation}
Moreover, both the infimum and supremum are attained, for all values of $\theta\in[0,1]$.
\end{theorem}

\begin{proof}
Fix $\theta\in[0,1]$.
From the inclusions~\eqref{eq:Thetaincl}, we conclude
\[
f_{qc}^l(\theta)\le f_{qa}^l(\theta)\le f_q^l(\theta)\le f_q^u(\theta)\le f_{qa}^u(\theta)\le f_{qc}^u(\theta).
\]
To find $f^l_{qc}(\theta)$, by Lemma~\ref{lem:symmC*}, we should find the supremum of values $\beta$ such that there exists a C$^*$-algebra $\Ac$
with faithful tracial state $\tau$ and 
with projections $A_0,A_1,A_2$ such that 
\begin{equation}\label{eq:vA}
\forall v,\quad\tau(A_v)=\theta,\qquad\tau(A_vA_{v+1})=\beta.
\end{equation}
By Cauchy-Schwarz, $\beta\le\theta$.
But taking $\Ac=\CC\oplus\CC$ with $A_v=1\oplus 0$ and an appropriate trace $\tau$ shows that $\beta=\theta$ occurs, and in a finite dimensional example.
Thus, we find $f_{qc}^l(\theta)=f_q^l(\theta)=\frac12$.

To find $f^u_{qc}(\theta)$,
again using Lenma~\ref{lem:symmC*}, we should find the infimum $\beta_0$ of values $\beta$ as described above.
Since $\Theta^s_{qc}$ is closed, this infimum is attained.
Thus, there exists a C$^*$-algebra $\Ac$ with tracial state $\tau$ and projections $A_0,A_1,A_2$ such that~\eqref{eq:vA} holds with $\beta=\beta_0$.
Morover, by the proof of Lemma~\ref{lem:symmC*},
we have that $\beta_0$ equals the infimum of $\frac{1}{3}\tau(AB+BC+CA)$ over all projections $A,B,C$ in some C$^*$-algebra with faithful tracial state
$\tau$ such that $\tau(A)=\tau(B)=\tau(C)=\theta$.
Thus, Corollary~\ref{three-projections-lemma} applies and the commutation relations
\[
[A_0,A_1+A_2]=[A_1,A_0+A_2]=[A_2,A_0+A_1]=0
\]
hold.
Thus, there is a representation of the universal C$^*$-algebra $\Afr$ considered in Proposition~\ref{prop:univC*}, sending $A$ to $A_0$, $B$ to $A_1$
and $C $ to $A_2$.
So, using Gelfand--Naimark--Segal representations, in order to find $\beta_0$, it suffices to consider tracial states (faithful or not) on $\Afr$.
In particular, $\beta_0$ is the minimum of all values of $\beta\ge0$ for which there exists a tracial state $\tau$ on $\Afr$ satisfying
\begin{equation}\label{eq:tauABC}
\tau(A)=\tau(B)=\tau(C)=\theta,\quad\tau(AB)=\tau(AC)=\tau(BC)=\beta.
\end{equation}
Since $\Afr$ is finite dimensional, we get $f^u_{qc}(\theta)=f^u_q(\theta)$.

An arbitrary tracial state of $\Afr$ is of the form
\[
\tau\left(\lambda_1\oplus\cdots\oplus\lambda_8\oplus\left(\begin{matrix}x_{11}&x_{12}\\x_{21}&x_{22}\end{matrix}\right)\right)
=\left(\sum_{j=1}^8t_j\lambda_j\right)+\frac s2(x_{11}+x_{22}),
\]
for some $t_1,\ldots,t_8,s\ge0$ satisfying $t_1+\cdots+t_8+s=1$.
The conditions~\eqref{eq:tauABC} become
\begin{gather*}
t_5+t_6+t_7+t_8+\frac s2=t_3+t_4+t_7+t_8+\frac s2=t_2+t_4+t_6+t_8+\frac s2=\theta, \\
t_7+t_8+\frac s8=t_6+t_8+\frac s8=t_4+t_8+\frac s8=\beta.
\end{gather*}
These are equivalent to
\begin{align*}
t_1&=1+3\beta-3\theta+\frac s8-t_8 \\
t_2=t_3=t_5&=\theta-2\beta-\frac s4+t_8 \\
t_4=t_6=t_7&=\beta-\frac s8-t_8.
\end{align*}
Thus, writing $t=t_8$, $\beta_0$ is the minimum value of $\beta$ such that there exist $s,t\ge0$ such that the inequalities
\[
1+3\beta-3\theta+\frac s8-t\ge0,\qquad\theta-2\beta-\frac s4+t\ge0,\qquad\beta-\frac s8-t\ge0.
\]
hold.
This is a linear programming problem.
We solved it by hand using the simplex method and also (to check) by using the Mathematica software platform~\cite{mma}.
The solution is,
\[
\beta_0=\begin{cases}
0,&0\le\theta\le\frac13 \\
\frac{3\theta-1}4,&\frac13\le\theta\le\frac12 \\
\frac{5\theta-2}4,&\frac12\le\theta\le\frac23 \\
2\theta-1,&\frac23\le\theta\le1,
\end{cases}
\]
which,
using $f^u_{qc}(\theta)=\frac12+\theta-\beta_0$, yields the values given in~\eqref{eq:flutmain-qc}.
\end{proof}

\section*{\it Acknowledgements} The authors would like to thank Tobias Fritz for several useful remarks.


\begin{thebibliography}{1}
\bibitem{CMNSW}
{\sc Cameron, P. ~J., Montanaro, A., Newman, M. ~W., Severini, S., and Winter, A.}
\newblock {On the quantum chromatic number of a graph}, 
\newblock {\em The Electronic Journal of Combinatorics}, 14 (2007), R81.

\bibitem{Collins2004}
{\sc Collins, D., and Gisin, N.}
\newblock {A relevant two qubit Bell inequality inequivalent to the CHSH  inequality}.
\newblock {\em Journal of Physics A: Mathematical and General 37}, 5 (2004), 1775--1787.

\bibitem{dykema2015}
{\sc Dykema, K.~J., and Paulsen, V.}
\newblock {Synchronous correlation matrices and Connes' embedding conjecture.}
\newblock {\em Journal of Mathematical Physics 57}, 1 (2016).

\bibitem{Fritz}
{\sc Fritz, T.}
\newblock {Tsirelson's problem and Kirchberg's conjecture.}
\newblock {\em Reviews in Mathematical Physics 24}, (2012), no.~05, 1250012.

\bibitem{Froissart1981}
{\sc Froissart, M.}
\newblock {Constructive generalization of Bell's inequalities.}
\newblock {\em Il Nuovo Cimento B (1971-1996) 64}, 2 (1981), 241--251.

\bibitem{junge2011etal}
{\sc Junge, M., Navascues, M., Palazuelos, C., Perez-Garcia, D., Scholz, V.~B., and Werner, R.~F.}
\newblock {Connes' embedding problem and Tsirelson's problem.}
\newblock {\em Journal of Mathematical Physics 52}, 1 (2011), 012102.

\bibitem{KPS}
{\sc Kim, S.-J., Paulsen, V.~I., and Schafhauser, C.},
\newblock {A synchronous game for binary constraint systems.}
\newblock {\em arxiv:1707.01016}, (2017).

\bibitem{LR}
{\sc Lackey, B., and Rodrigues, N.}
\newblock {Nonlocal games, synchronous correlations, and Bell inequalities}
\newblock {preprint, July 10, 2017}

\bibitem{NGHA}
{\sc Navascu{\'e}s, M., Guryanova, Y., Hoban, M. J., and Ac{\'i}n, A.}
\newblock {Almost quantum correlations.}
\newblock {\em Nat. Commun. 6:6288,} (2015)

\bibitem{NPA}
{\sc Navascu{\'e}s, M., Pironio, S., and Ac{\'i}n, A.}
\newblock {A convergent hierarchy of semidefinite programs characterizing the set of quantum correlations}
\newblock {\em New Journal of Physics 10,} (2008), no.~7, 073013.

\bibitem{ozawa2013}
{\sc Ozawa, N.}
\newblock {About the Connes' embedding conjecture.}
\newblock {\em Japanese Journal of Mathematics 8}, 1 (2013), 147--183.


\bibitem{Pal2010}
{\sc P\'al, K.~F., and V\'ertesi, T.}
\newblock {Maximal violation of a bipartite three-setting, two-outcome bell inequality using infinite-dimensional quantum systems.}
\newblock {\em Phys. Rev. A 82}, (2010), 022116.


\bibitem{paulsen2016}
{\sc Paulsen, V.~I., Severini, S., Stahlke, D., Todorov, I.~G., and Winter, A.}
\newblock {Estimating quantum chromatic numbers.}
\newblock {\em Journal of Functional Analysis 270}, 6 (2016), 2188 -- 2222.

\bibitem{PT}
{\sc Paulsen, V.~I., and Todorov, I. ~G.}
\newblock {Quantum chromatic numbers via operator systems.}
\newblock {\em The Quarterly Journal of Mathematics 66}, 2 (2015), 677.

\bibitem{Sl1}
{\sc Slofstra, W.,}
\newblock {Tsirelson's problem and an embedding theorem for groups arising from non-local games.}
\newblock {\em arXiv:1606.03140.}, (2016).

\bibitem{Sl2}
{\sc Slofstra, W.,}
\newblock {The set of quantum correlations is not closed.}
\newblock {\em arXiv:1703.08618.}, (2017).

\bibitem{Ts1}
{\sc Tsirelson, B. ~S.}
\newblock {Some results and problems on quantum Bell-type inequalities.}
\newblock {\em Hadronic Journal Supplement 8}, 4 (1993), 329 -- 345.

\bibitem{Ts2}
{\sc Tsirelson, B. ~S.}
\newblock {Bell inequalities and operator algebras.}
\newblock {\url{http://www.imaph.tu-bs.de/qi/problems/33.html}}, (2006).

\bibitem{Vidick2011}
{\sc Vidick, T. and Wehner, S.}
\newblock {More nonlocality with less entanglement.}
\newblock {\em Phys. Rev. A 83}, 5 (2011), 052310.


\bibitem{mma}
{\sc Wofram Research Inc.}
\newblock {\em Mathematica, Version 11.0.}
\newblock {Wolfram Research Inc., Champaign, IL}, 2016.


\end{thebibliography}
\end{document}